\documentclass[10pt,reqno]{amsart}
\usepackage{amsmath,amsopn,amssymb,amsthm}
\usepackage[cp1251]{inputenc}
\usepackage[english]{babel}
\usepackage[pagebackref,breaklinks=true,colorlinks=true,linkcolor=blue,citecolor=blue,urlcolor=blue]{hyperref}
\usepackage{graphicx}%
\usepackage{color}

\newtheorem{theorem}{Theorem}[section]

\newtheorem{example}[theorem]{Example}

\newtheorem{lemma}[theorem]{Lemma}

\newtheorem{proposition}[theorem]{Proposition}

\renewenvironment{proof}[1][Proof]{\noindent\textbf{#1.} }{\ \rule{0.5em}{0.5em}}

\begin{document}

\title[A Counter Example for Homogeneity Conjecture]{A Counter Example for  Homogeneity Conjecture}
\author{Ming Xu}
\address[Ming Xu] {School of Mathematical Sciences,
Capital Normal University,
Beijing 100048,
P. R. China}
\email{mgmgmgxu@163.com}
\author{Shaoqiang Deng}
\address[Shaoqiang Deng] {School of Mathematical Sciences and LPMC, Nankai University, Tianjin 300071, P. R. China}
\email{dengsq@nankai.edu.cn}
\thanks{S. Deng is the corresponding author. This paper is supported by National Natural Science Foundation of China (No. 12131012).}
\begin{abstract}
We  construct a  counter example to show that the Homogeneity Conjecture,  proposed by J.A. Wolf in 1960, is not true.  To be precise, we prove that on the group $\mathrm{Sp}(2)$, there exists a left invariant Riemannian metric and a cyclic subgroup
$\Gamma$ of order $(2n+1)$, such that the left translation of each element of $\Gamma$ on $\mathrm{Sp}(2)$ is a Clifford Wolf translation, and
the Riemannian quotient $\Gamma\backslash \mathrm{Sp}(2)$ is not homogeneous.

\vbox{}

\end{abstract}

\maketitle

\section{Introduction}
It is our goal of this article to solve a long standing open problem in Riemannian geometry.
Recall that an isometry $\rho$ of a Riemannian manifold $M$ is called a {\it Clifford Wolf translation} (or  a {\it CW-translation}) if $d(x,\rho (x))$ is a constant function on $M$, where $d$ is the distance function of $M$. In other words, an isometry of a Riemmanian manifold is a CW-translation if it moves each point  the same distance. The significance of CW-translations was  revealed by the following fundamental observation (see Theorems 1 and 2 in \cite{Wo1960}):

\begin{proposition}\label{prop-1}
 Let $\Gamma$ be a discontinuous group of fixed point free isometries of the connected Riemannian manifold $M$, such that the quotient map from $M$ to $\Gamma\backslash M$ is locally isometric,  and the induced metric on $M'=\Gamma\backslash M$ is homogeneous. Then each
element in $\Gamma$ is a CW-translation on $M$, and $M$ is also  homogeneous.
\end{proposition}

For simplicity, we will call the Riemannian manifold $\Gamma\backslash M$ in Proposition \ref{prop-1} a {\it Riemannian quotient}.
J.A. Wolf conjectured that the  converse  of Proposition \ref{prop-1} is also true, in the study of complete Riemannian manifolds of constant curvature (i.e., space forms) \cite{Wo2011}.
This conjecture is referred to as the {\it Homogeneity Conjecture} in the literature, and can be stated as follows:
\smallskip

\noindent{\bf Homogeneity Conjecture}\quad {\it Let $M$ be a connected simply connected homogeneous Riemannian manifold and $\Gamma$  a discontinuous group of isometries of $M$, such that the projection $\pi:M\rightarrow M'=\Gamma\backslash M$ is a Riemannian covering. Then $M'$ is a homogeneous Riemannian manifold if
and only if each element in $\Gamma$ is a CW-translation.}
\bigskip

This Conjecture has been studied for several decades. In the early 1960's, J.A. Wolf  proved that the conjecture is true when $M$ is a space form \cite{Wo1960,Wo1961}. Moreover, he also verified the conjecture in the case that $M$ is Riemannian symmetric \cite{Wo1962} (see also \cite{Fr1963,Oz1974} for alternative approaches of H. Freudenthal and V. Ozols), or
 $M$ has negative curvature \cite{Wo1964}. In recent years, the Homogeneity Conjecture has been proved to be true in the cases that  $M$ is a Stieffel manifold, or $M$ has a positive Euler characteristic,  or $M$ is certain group manifold \cite{Wo2017,Wo2023}.
 In \cite{DW2012}, S. Deng and J.A. Wolf considered an analog of the Homogeneity Conjecture in Finsler geometry, and proved it for symmetric Finsler spaces.

The Homogeneity Conjecture has fostered  many interesting  research  in homogeneous geometry. For example, V.N. Berestovskii and Yu.G. Nikonorov  classified connected simply connected Clifford-Wolf homogeneous Riemannian manifolds \cite{BN2008,BN2008-2,BN2009}.
Their results were generalized to Finsler geometry  by S. Deng and M. Xu \cite{DX2014,DX2014-2,DX2016,XD2013,XD2015,XD2015-2}.
The interrelation between CW-translations and CK-vector fields (i.e., Killing vector fields of constant length) is the key  in  related studies, and it will also be the major guideline for our consideration in this paper.
Bounded isometries and bounded Killing vector fields were explored in
\cite{Ti1964,Wo1964,Wo2017,XN2021}. In particular, J.A. Wolf and M. Xu
classified some special CK-vector fields on normal homogeneous Riemannian manifolds \cite{XW2016}. The work of Wolf-Xu
 generalizes the results in \cite{Wo1962}, and provides a crucial technique
in this paper. More details  for the study of CK-vector fields on homogeneous Riemannian manifolds or
pseudo-Riemannian normal homogeneous manifolds can be found in  \cite{Ni2013,Ni2015,WPX2017}.

Nevertheless, the  conjecture remains open before our consideration in this paper. In fact, most of the researchers in this field incline to believe that the conjecture is true, since
the existence of CW-translations does impose strong rigidity to  homogeneous Riemannian metrics. The difficulty in dealing with this problem lies in the fact that
CW-translations are rare and it is very hard to verify that an explicit isometry is a CW-translation. In this paper, we will construct
a counter example, which completely solves this long standing open problem.
The main result of this paper can be stated as follows:

\begin{theorem}\label{main-thm}
On the Lie group $\mathrm{Sp}(2)$, there exists a left invariant Riemannian metric and a cyclic subgroup
$\Gamma$ of order $2n+1$,  consisting of CW-translations,
such that the Riemannian quotient $\Gamma\backslash \mathrm{Sp}(2)$ is not homogeneous.
\end{theorem}

Here $\Gamma\subset \mathrm{Sp}(2)$ consists of CW-translations means that the left translation $L(g)$ on $\mathrm{Sp}(2)$,
$g'\mapsto gg'$, for each $g\in\Gamma$, is a CW-translation, with respect to
the specified left invariant metric.


The strategy for finding the counter example  originates from the exploration of right invariant CK-vector fields on a Lie group endowed with a left invariant Riemannian metric (we will call it a {\it left invariant Riemannian Lie group} for simplicity). Generally speaking, when the Lie group is compact, connected, simply connected and simple, and the metric is not bi-invariant, nontrivial right invariant CK-vector fields are very rare. Surprisingly, they do exist in a few occasions. Using some CK-vector fields on
normal homogeneous manifolds \cite{XW2016}, nontrivial right invariant CK-vector fields can be found on $\mathrm{SU}(2n)$, $\mathrm{Spin}(2n)$ and $\mathrm{Spin}(7)$, where certain Riemannian metrics are endowed, which are left invariant, but not bi-invariant (see Theorem \ref{thm-1} and Lemma \ref{lemma-4}). In particular, that metric and CK-vector field on $\mathrm{Spin}(7)$ can be inherited by the subgroup $\mathrm{Sp}(2)$ of $\mathrm{Spin}(7)$. To be precise,
we construct a smooth family of left invariant metrics $\langle\cdot,\cdot\rangle_s$, $s>0$ on $\mathrm{Sp}(2)$, which are bi-invariant if and only if $s=1$, and share a common nonzero right invariant CK-vector field, induced by a nonzero vector $v\in\mathfrak{sp}(2)$ (see Example \ref{example-1}). In the one-parameter subgroup generated by that $v$, we can find a cyclic subgroup $\Gamma$ of order $(2n+1)$.
On one hand, the Riemannian quotients $(\Gamma\backslash \mathrm{Sp}(2),\langle\cdot,\cdot\rangle_s)$ is not homogeneous when $s\neq1$ (see Lemma \ref{lemma-11}).
On the other hand, the
left translations of $\Gamma$ on $(\mathrm{Sp}(2),\langle\cdot,\cdot\rangle_s$
consist of CW-translations when $s$ is sufficiently close to $1$ (see Lemma \ref{lemma-12}).
To summarize, we just choose $s\neq1$ which is sufficiently close to $1$,
then the counter example is found.

This paper is organized as follows. In Section 2, we present some preliminaries on homogeneous geometry and Lie theory, and set the notations which will be used in this paper. In Section 3, we introduce CW-translations and CK-vector fields, and discuss right invariant CK-vector fields on Lie groups. In Section 4, we construct the counter example, which proves Theorem \ref{main-thm}.
\section{Preliminaries}

\subsection{Left invariant and bi-invariant Riemannian metrics}
Let $G$ be a compact connected simple Lie group. Then the  Lie algebra $\mathfrak{g}=\mathrm{Lie}(G)$ is also compact simple.
Let $\langle\cdot,\cdot\rangle$ be an inner product on $\mathfrak{g}$. By
left translations, $\langle\cdot,\cdot\rangle$ naturally induces a {\it left invariant metric} on $G$.
By the compactness of $\mathfrak{g}$,
there exists a bi-invariant (i.e.,
$\mathrm{Ad}(G)$-invariant) inner product $\langle\cdot,\cdot\rangle_\mathrm{bi}$ on $\mathfrak{g}$, which induces
a {\it bi-invariant metric} on $G$. Since  $\mathfrak{g}$ is  simple, the inner product $\langle\cdot,\cdot\rangle_{\mathrm{bi}}$ must be a negative multiple of the Killing form, so up to scalars, it is unique.
For the sake of simplicity, we will use the same notation to denote the inner product and the corresponding metric. In particular,   we will say that  $(G,\langle\cdot,\cdot\rangle)$ is a {\it left invariant Riemannian Lie group}, and  $(G,\langle\cdot,\cdot\rangle_\mathrm{bi})$ a {\it bi-invariant Riemannian Lie group}.

Let $I(G,\langle\cdot,\cdot\rangle)$ be the full group of isometries
of $(G,\langle\cdot,\cdot\rangle)$, and
$I_0(G,\langle\cdot,\cdot\rangle)$  the identity component of $I(G,\langle\cdot,\cdot\rangle)$. It is well known
that $L(G)\subset I_0(G,\langle\cdot,\cdot\rangle)\subset L(G)R(G)$ (see Theorem 1 in \cite{OT1976}). Here $L(G)$ and $R(G)$ are the Lie groups of all left translations and the group of all right translations on $G$, and their Lie algebras
can be canonically identified with the space of all right invariant vector fields and the space of all left invariant vector fields, respectively.
This observation implies that
$I_0(G,\langle\cdot,\cdot\rangle)$ is locally isomorphic to a product Lie group $G\times K$, where $G$ consists of  all left translations, and
$$K=\{g\in G|\langle\mathrm{Ad}(g)u,\mathrm{Ad}(g)v\rangle
=\langle u,v\rangle,\forall u,v\in\mathfrak{g}\}$$
corresponds to all isometric right translations.
In particular, $I_0(G,\langle\cdot,\cdot\rangle_{\mathrm{bi}})=L(G)R(G)\cong
G\times G/\Delta C(G)$, where $\Delta C(G)=\{(g,g)| g\in C(G)\}$ is a discrete normal subgroup of $G\times G$.

For any discrete subgroup $\Gamma$ of the left invariant Riemannian Lie group $(G,\langle\cdot,\cdot\rangle)$, the smooth coset space $\Gamma\backslash G=\{\Gamma g\ |\ \forall g\in G\}$ can be endowed with a metric, which will be denoted by the same $\langle\cdot,\cdot\rangle$, such that
the quotient map from $G$ to $\Gamma\backslash G$ is locally isometric. For simplicity,
We call $(\Gamma\backslash G,\langle\cdot,\cdot\rangle)$ a {\it Riemannian quotient}. The Riemannian quotient of a left invariant Riemannian Lie group may not be homogeneous in general.

\subsection{Homogeneous and normal homogeneous Riemannian metrics}

Let $G$ be compact connected simple Lie group, and $H$ a closed connected subgroup of $G$. Denote by $\mathfrak{g}$ and $\mathfrak{h}$ their Lie algebras respectively. A direct sum decomposition $\mathfrak{g}=\mathfrak{h}+\mathfrak{m}$ of $\mathfrak{g}$, where $\mathfrak{m}$ is a subspace of $\mathfrak{g}$, is called a {\it reductive decomposition} for the homogeneous manifold $G/H$, if
it is $\mathrm{Ad}(H)$-invariant, or equivalently, in the Lie algebra level, if $[\mathfrak{h},\mathfrak{m}]\subset\mathfrak{m}$. The compactness of $H$
guarantees the existence of reductive decompositions. In general, we will take $\mathfrak{m}$ to be the $\langle\cdot,\cdot\rangle_{\mathrm{bi}}$-orthogonal complement of $\mathfrak{h}$. In this case,   the corresponding reductive decomposition will usually be referred to as
{\it bi-invariant orthogonal}.

Let $\langle\cdot,\cdot\rangle$ be an inner product on $\mathfrak{m}$ which is $\mathrm{Ad}(H)$-invariant, i.e.,  $$\langle \mathrm{Ad}(g)u,\mathrm{Ad}(g)v\rangle=\langle u,v\rangle,\quad\forall g\in H,\ u,v\in\mathfrak{m}.$$
 By the left $G$-actions, it induces a {\it homogeneous} (or $G$-invariant) Riemannian metric on $G/H$. This correspondence is one-to-one. In the case that $\langle\cdot,\cdot\rangle=\langle\cdot,\cdot\rangle_{\mathrm{bi}}
|_{\mathfrak{m}\times\mathfrak{m}}$, the corresponding homogeneous metric
is called {\it normal homogeneous}.

%

\subsection{Conventions and notations for roots and root planes}
Let $G/H$ be a homogeneous manifold, where $G$ is a compact connected simple Lie group, and $H$  a closed connected subgroup of $G$.
Let $\mathfrak{g}=\mathfrak{h}+\mathfrak{m}$ be a bi-invariant orthogonal decomposition for $G/H$.
To study the roots and root planes of $\mathfrak{g}=\mathrm{Lie}(G)$ and $\mathfrak{h}=\mathrm{Lie}(H)$, we adopt the following conventions and notations from \cite{XW2016}.

Let  $\mathfrak{t}$ be a Cartan subalgebra  of $\mathfrak{g}$ such that
$\mathfrak{t}\cap\mathfrak{h}$ is a Cartan subalgebra of $\mathfrak{h}$.
Then we have the root plane decompositions:
\begin{equation}\nonumber\label{001}
\mathfrak{g}=\mathfrak{t}+\sum_{\alpha\in\Delta_\mathfrak{g}}
\mathfrak{g}_{\pm\alpha}\quad\mbox{and}\quad\mathfrak{h}=\mathfrak{t}\cap
\mathfrak{h}
+\sum_{\alpha'\in\Delta_\mathfrak{h}}\mathfrak{h}_{\alpha'},
\end{equation}
here $\Delta_\mathfrak{g}$ and $\Delta_\mathfrak{h}$ are the root systems
of $\mathfrak{g}$ with respect to $\mathfrak{t}$ and $\mathfrak{h}$ with respect to $\mathfrak{h}\cap\mathfrak{t}$, respectively. Fix a bi-invariant inner product $\langle\cdot,\cdot\rangle_{\mathrm{bi}}$ on $\mathfrak{g}$.
We can regard $\Delta_\mathfrak{g}$ and $\Delta_\mathfrak{h}$ as subsets of $\mathfrak{t}$ and $\mathfrak{t}\cap\mathfrak{h}$,  respectively. More precisely, for any root $\alpha\in\mathfrak{t}$ of $\mathfrak{g}$ and  $\alpha'\in\mathfrak{t}\cap\mathfrak{h}$ of $\mathfrak{h}$, there exist a basis $\{u,v\}$ of $\mathfrak{g}_{\pm\alpha}$ and a basis
$\{u',v'\}$ of $\mathfrak{h}_{\pm\alpha'}$, such that
\begin{eqnarray*}
& &[x,u]=\langle x,\alpha\rangle_{\mathrm{bi}}v,\quad [x,v]=-\langle x,\alpha\rangle_{\mathrm{bi}}u,\quad\forall x\in\mathfrak{t},
\end{eqnarray*}
and
\begin{eqnarray*}
& &[x',u']=\langle x',\alpha'\rangle_{\mathrm{bi}}v', \quad
[x',v']=-\langle x',\alpha'\rangle_{\mathrm{bi}}u',\quad
\forall x'\in\mathfrak{t}\cap\mathfrak{h}.
\end{eqnarray*}

As an example, we consider the case $(\mathfrak{g},\mathfrak{h})=(B_3,G_2)$ corresponding to the homogeneous sphere $S^7=\mathrm{Spin}(7)/\mathrm{G}_2$.
Their roots and root planes have been discussed in \cite{XW2016}.

%
\label{section-2-4}

We may adjust
the bi-invariant inner product $\langle\cdot,\cdot\rangle_\mathrm{bi}$ on $\mathfrak{g}$ with a
suitable scalar, such that the root system of $\mathfrak{g}$ can be canonically presented as $$\Delta_\mathfrak{g}=\{\pm e_i\pm e_j,\ \forall 1\leq i<j\leq 3;\quad \pm e_i,\forall 1\leq i\leq 3\},$$
where $\{e_1,e_2,e_3\}$ is an orthonormal basis of $\mathfrak{t}$. Since
$\mathfrak{t}\cap\mathfrak{h}$ and $\mathfrak{t}\cap\mathfrak{m}$ are linearly spanned by $\{e_1-e_2,e_2-e_3\}$ and $e_1+e_2+e_3$ respectively,
the root system of $\mathfrak{h}$ coincides with the following subset:
$$\Delta_\mathfrak{h}=\{\pm (e_i-e_j),\ \forall 1\leq i<j\leq 3;\quad\tfrac13(e_1+e_2+e_3)-e_i,\forall 1\leq i\leq 3\}.$$

The root planes of $\mathfrak{h}$ are described in the following lemma.

\begin{lemma}\label{lemma-1}$(1)$ For $1\leq i\ne j\leq 3$, the root plane $\mathfrak{g}_{\pm(e_i-e_j)}$ of $\mathfrak{g}$ coincides with the root plane $\mathfrak{h}_{\pm(e_i-e_j)}$ of $\mathfrak{h}$.

$(2)$ For $\{i,j,k\}=\{1,2,3\}$, we have the linear decomposition
$$\mathfrak{g}_{\pm(e_i+e_j)}+\mathfrak{g}_{\pm e_k}=\mathfrak{h}\cap(\mathfrak{g}_{\pm(e_i+e_j)}+\mathfrak{g}_{\pm e_k})
+\mathfrak{m}\cap(\mathfrak{g}_{\pm(e_i+e_j)}+\mathfrak{g}_{\pm e_k}),$$
in which $\mathfrak{h}\cap(\mathfrak{g}_{\pm(e_i+e_j)}+\mathfrak{g}_{\pm e_k})$ coincides with the root plane $\mathfrak{h}_{\pm(\tfrac13e_i+\tfrac13e_j-\tfrac23e_k)}$ of $\mathfrak{h}$.

$(3)$ For $1\leq i\neq j\leq 3$,
$\mathfrak{h}\cap\mathfrak{g}_{\pm e_i}=\mathfrak{h}\cap\mathfrak{g}_{\pm(e_i+e_j)}=\{0\}$.

$(4)$ For $1\leq i\neq j\leq 3$, $\mathfrak{m}\cap\mathfrak{g}_{\pm e_i}= \mathfrak{m}\cap\mathfrak{g}_{\pm(e_i+e_j)}=\{0\}$.
\end{lemma}
\begin{proof} (1) and (2) follow immediately after the root plane decompositions of $\mathfrak{g}$ and $\mathfrak{h}$. See Section 2 in \cite{XW2016} for more details.

 For (3), we will only verify  $\mathfrak{g}_{\pm(e_1+e_2)}\cap\mathfrak{h}=\{0\}$, because other statements can be proved similarly.
Assume conversely that $\mathfrak{g}_{\pm(e_1+e_2)}\cap\mathfrak{h}\neq\{0\}$. Then $\mathfrak{g}_{\pm(e_1+e_2)}$ is linearly spanned by $\mathfrak{g}_{\pm(e_1+e_2)}\cap\mathfrak{h}$ and $[e_2-e_3,\mathfrak{g}_{\pm(e_1+e_2)}\cap\mathfrak{h}]\subset\mathfrak{h}$.
This implies that $\mathfrak{g}_{\pm(e_1+e_2)}\subset\mathfrak{h}$. Then we have
$e_1+e_2\in[\mathfrak{g}_{\pm(e_1+e_2)},
\mathfrak{g}_{\pm(e_1+e_2)}]\subset\mathfrak{h}$, which
contradicts the fact that $\mathfrak{t}\cap\mathfrak{h}=\mathbb{R}(e_1-e_2)+\mathbb{R}(e_2-e_3)$.

(4) is the corollary of (3) and the bi-invariant orthogonality.
\end{proof}

\section{CW-translations and CK-vector fields}
\subsection{Examples of CW-translations and CK-vector fields}
Now we recall some results about CW-translations and CK-vector fields. In \cite{Wo1962}, J.A. Wolf initiated the study of CW-translations
on Riemannian symmetric spaces. He showed that parallel transformations on Euclidean spaces, antipodal maps on standard spheres, and left and right translations on a bi-invariant Riemannian  Lie group $(G,\langle\cdot,\cdot\rangle_\mathrm{bi})$ are all CW-translations. Now we summarize some straightforward criterions for CW-translations.

Recall that a Killing vector field $V$ on a Riemannian manifold $(M,\langle,\rangle)$ is called a {\it CK-vector field}
if it is of constant length, i.e.,  if $\langle V(x),V(x)\rangle$ is a constant function of $x\in M$. Equivalently,  a Killing vector field $V$ is a CK-vector field if and only if
any integral curve of $V$ is a geodesic.

CK-vector fields are closely related  to CW-translations. In fact,  on a compact or homogeneous Riemannian manifold,  a CK-vector field generates a local one-parameter subgroup which consists of CW-translations, and conversely, any
CW-translation which is sufficiently close to the identity map can be generated by a CK-vector field (see Section 4 in \cite{BN2009}).
%

For example, left invariant and right invariant vector fields on
$(G,\langle\cdot,\cdot\rangle_{\mathrm{bi}})$ are CK-vector fields. The corresponding CW-translations
 are right and left translations, respectively.
This assertion can be verified using the following
criterion.
\begin{lemma}\label{lemma-2}
Let $\rho$ (or $V$) be an isometry (a Killing vector field) on a connected Riemannian manifold $M$. If the centralizer of $\rho$ (or $V$) in the isometry group acts transitively on $M$, then $\rho$ (or $V$) is a CW-translation (a CK-vector field).
\end{lemma}

Notice that the condition in Lemma \ref{lemma-2} implies that the Riemannian manifold $M$ must be homogeneous. Moreover,   the lemma can  be applied to very few cases. To study general CK-vector fields, we need a better criterion (see Lemma \ref{lemma-3} below).

\subsection{CK-vector fields on homogeneous Riemannian manifolds}

Let  $G$ be a connected compact  simple Lie group, and $H$  a closed connected subgroup of $G$. Consider the coset space $G/H$. Denote
by $\mathfrak{g}=\mathfrak{h}+\mathfrak{m}$ the bi-invariant orthogonal
decomposition for $G/H$. As mentioned above, we use   $\langle\cdot,\cdot\rangle$ to denote an $\mathrm{Ad}(H)$-invariant inner product on $\mathfrak{m}$, as well as
the homogeneous metric  induced on $G/H$. Any vector $v\in\mathfrak{g}$
generates a Killing vector field $X$ on the homogeneous
Riemannian manifold $(G/H,\langle\cdot,\cdot\rangle)$ defined by
\begin{equation*}
X(gH)=\tfrac{d}{dt}|_{t=0}(\exp tv g) H,\quad\forall g\in G.
\end{equation*}
The following well known criterion can be applied to determine if $X$ is a CK-vector field.

%
%
%

\begin{lemma}\label{lemma-3}
The vector $v\in\mathfrak{g}$ generates a CK-vector field on the homogeneous Riemannian manifold $(G/H,\langle\cdot,\cdot\rangle)$ if and only if
\begin{equation}\nonumber
\langle\mathrm{pr}_\mathfrak{m}(\mathrm{Ad}(g)v),\mathrm{pr}_\mathfrak{m}
(\mathrm{Ad}(g)v)\rangle\equiv \mathrm{const\,\, for} \,\, g\in G,
\end{equation}
where $\langle\cdot,\cdot\rangle$ is the inner product and $\mathrm{pr}_\mathfrak{m}:\mathfrak{g}\rightarrow\mathfrak{m}$
is the linear projection with respect to the bi-invariant reductive decomposition. Moreover, if $v\in\mathfrak{g}$ generates a CK-vector field on
$(G/H,\langle\cdot,\cdot\rangle)$, then so does $\mathrm{Ad}(g)v$ for any $g\in G$.
\end{lemma}

In the special case that  $(G/H,\langle\cdot,\cdot\rangle)$ is normal homogeneous
and $0<\dim H <\dim G$,  nonzero CK-vector fields $v\in\mathfrak{g}$ are classified in \cite{XW2016}. The following main theorem in \cite{XW2016} implies that, in this case,
the existence of a nonzero CK-vector field $v\in\mathfrak{g}$ imposes
strong rigidity of the metric.

\begin{theorem}\label{thm-1}
Let $G$ be a  connected compact simple Lie group and $H$
a closed subgroup with $0<\dim H<\dim G$. Fix a normal Riemannian metric
on $M=G/H$. Suppose there is a  nonzero vector $v\in\mathfrak{g}$ generating a CK-vector field on $M=G/H$. Then $M$ is a complete locally symmetric Riemannian manifold. Moreover,  the universal Riemannian covering of $M$ is either an odd-dimensional sphere of constant curvature, or the Riemannnian symmetric space
$\mathrm{SU}(2n)/\mathrm{Sp}(n)$.
\end{theorem}

Using the method in  the proof of Theorem \ref{thm-1} in \cite{XW2016}, we can determine
 all the CK-vector fields $v\in\mathfrak{g}$ in Theorem \ref{thm-1}.
In particular, we have the following example (see Section 6B in \cite{XW2016}).
\begin{example}\label{example-3}
The vector $v=e_1\in\mathfrak{g}=B_3$ generates a CK-vector field on the normal homogeneous Riemannian $7$-sphere $G/H=\mathrm{Spin}(7)/\mathrm{G}_2$.
\end{example}
Notice that here we have applied the conventions and notations in Section \ref{section-2-4}.
\subsection{CK-vector fields on left invariant Riemannian Lie groups}

Let $(G,\langle\cdot,\cdot\rangle)$ be a left invariant Riemannian compact connected simple Lie group. By Lemma \ref{lemma-2}, all left invariant fields on $(G,\langle\cdot,\cdot\rangle)$ are CK-vector fields. Therefore it is natural to study  the right invariant ones and search for CK-vector fields among them.
The manifold $(G,\langle\cdot,\cdot\rangle)$ can be viewed as the homogeneous Riemannian manifold $(G/H=G/\{e\},\langle\cdot,\cdot\rangle)$, associated with the bi-invariant orthogonal decomposition $\mathfrak{g}=\mathfrak{h}+\mathfrak{m}=0+\mathfrak{g}$.
As an immediate corollary of Lemma \ref{lemma-3}, the following lemma
 gives a condition for a vector $v\in\mathfrak{g}$ to generate a right invariant CK-vector field on $(G,\langle\cdot,\cdot\rangle)$.

\begin{lemma}\label{lemma-6}
The right invariant vector field
generated by $v\in\mathfrak{g}$ is a CK-vector field on $(G,\langle\cdot,\cdot\rangle)$ if and only if
\begin{equation*}
\langle\mathrm{Ad}(g)v,\mathrm{Ad}(g)v\rangle\equiv\mathrm{const},\quad \forall g\in G,
\end{equation*}
where $\langle\cdot,\cdot\rangle$ is the inner product on $\mathfrak{g}$.
Moreover, if $v\in\mathfrak{g}$ generates a CK-vector field on $(G,\langle\cdot,\cdot\rangle)$, then so does $\mathrm{Ad}(g)v$ for any $g\in G$.
\end{lemma}

 In the case that $\langle\cdot,\cdot\rangle$ is not bi-invariant, nonzero right invariant CK-vector fields are generally  very rare. On the other hand, the following lemma, together with Theorem \ref{thm-1}, indicates that  they do exist.

\begin{lemma}\label{lemma-4}
Let $G$ be a connected compact  simple Lie group, and $H$ a closed connected subgroup of $G$, with $0<\dim H<\dim G$. Let $\mathfrak{g}=\mathfrak{h}+\mathfrak{m}$ be the bi-invariant orthogonal reductive decomposition for $G/H$. If the nonzero vector $v\in\mathfrak{g}$ generates a CK-vector field for the normal homogeneous Riemannian metric on $G/H$, then it also generates a right invariant CK-vector field for the following left invariant metrics on $G$:
$$\langle\cdot,\cdot\rangle_{\mathrm{bi}}|_{\mathfrak{h}\times\mathfrak{h}}
\oplus s\langle\cdot,\cdot\rangle_{\mathrm{bi}}|_{\mathfrak{m}\times\mathfrak{m}},
\quad\forall s>0.$$
\end{lemma}
The following lemma can be used to  find out more right invariant CK-vector fields.
\begin{lemma}\label{lemma-5}
Let $\mathfrak{l}$ be a subalgebra of $\mathfrak{g}=\mathrm{Lie}(G)$
and   $L$ be the corresponding connected Lie subgroup of $G$. If $v\in\mathfrak{l}$
generates a right invariant CK-vector field on the left invariant Riemannian
Lie group $(G,\langle\cdot,\cdot\rangle)$, then it also generates
a right invariant CK-vector field on the left invariant Riemannian Lie group
$(L,\langle\cdot,\cdot\rangle|_{\mathfrak{l}\times\mathfrak{l}})$.
\end{lemma}
Both Lemma \ref{lemma-4} and Lemma \ref{lemma-5} are immediate corollaries of
Lemma \ref{lemma-6} (which does not really need the compactness condition), so the proofs are omitted.

\begin{example}\label{example-1}Let $\mathfrak{g}=B_3$,
$\mathfrak{h}=G_2$, and $\mathfrak{m}$ be as in Section \ref{section-2-4}. Consider the subalgebra
$$\mathfrak{sp}(2)=\mathbb{R}e_1+\mathbb{R}e_2+\mathfrak{g}_{\pm e_1}+\mathfrak{g}_{\pm e_2}+\mathfrak{g}_{\pm(e_1+e_2)}+
\mathfrak{g}_{\pm(e_1-e_2)}$$
of $\mathfrak{g}$, which generates the subgroup $\mathrm{Sp}(2) =\mathrm{Spin}(5)$ of $G=\mathrm{Spin}(7)$. Let
$$\langle\cdot,\cdot\rangle_s=
(\langle\cdot,\cdot\rangle_{\mathrm{bi}}|_{\mathfrak{h}\times\mathfrak{h}}
\oplus s\langle\cdot,\cdot\rangle_{\mathrm{bi}}|_{\mathfrak{m}\times\mathfrak{m}})
|_{\mathfrak{sp}(2)\times\mathfrak{sp}(2)},
\quad\forall s>0,$$
be a family of inner products on $\mathfrak{sp}(2)$ and consider the corresponding left invariant metrics on $\mathrm{Sp}(2)$.
 Then by Example \ref{example-3}, Lemma \ref{lemma-4} and Lemma \ref{lemma-5}, $v=e_1\in\mathfrak{sp}(2)$ generates a right invariant CK-vector field on
$(\mathrm{Sp}(2),\langle\cdot,\cdot\rangle_s)$. Notice that when $s\neq1$, $\langle\cdot,\cdot\rangle_s$ is not bi-invariant.
\end{example}

\section{Counter examples for Homogeneity Conjecture}
%
\label{sub-section-4-2}
Through out this section, we assume that $(G,\langle\cdot,\cdot\rangle_{\mathrm{bi}})$ is the Lie group $\mathrm{Spin}(7)$ endowed with a bi-invariant Riemannian metric. Let
$H=\mathrm{G}_2$ be the connected closed subgroup of $G$, such that the Lie algebras $\mathfrak{g}$ and $\mathfrak{h}$ of $G$ and $H$, respectively, and inner product $\langle\cdot,\cdot\rangle_{\mathrm{bi}}$ are the same as  in Section \ref{section-2-4}. Let $(\mathrm{Sp}(2),\langle\cdot,\cdot\rangle_s)$
with $s>0$ be the left invariant Riemannian $\mathrm{Sp}(2)$ in Example \ref{example-1}.
Denote by $C$ the centralizer of $v=e_1\in\mathfrak{sp}(2)$ in $\mathrm{Sp}(2)$, and by $K_s$ the connected closed subgroup of $\mathrm{Sp}(2)$ such that $I_0(\mathrm{Sp}(2),\langle\cdot,\cdot\rangle_s)
=L(\mathrm{Sp}(2))R(K_s)$.

\begin{lemma}\label{lemma-7}
For any $s\neq 1$,
the action of $C\times K_s$ on $\mathrm{Sp}(2)$, defined by $((g_1,g_2),g)\in ((C\times K_s)\times \mathrm{Sp}(2))\mapsto g_1gg_2^{-1}$, is not transitive.
\end{lemma}

\begin{proof}Assume $s>1$. Lemma \ref{lemma-1} implies that
$\mathfrak{sp}(2)\cap\mathfrak{h}=\mathbb{R}(e_1-e_2)+
\mathfrak{g}_{\pm(e_1-e_2)}\neq0$.
Then we have
$\min_{u\in\mathfrak{sp}(2)\backslash\{0\}}\tfrac{\langle u,u\rangle_s}{\langle u,u\rangle_\mathrm{bi}}=1$,
which is achieved exactly by the punctured subspace $\mathfrak{sp}(2)\cap\mathfrak{h}\backslash\{0\}$.
Since the $\mathrm{Ad}(K_s)$-actions preserve both $\langle\cdot,\cdot\rangle_s$
and $\langle\cdot,\cdot\rangle_\mathrm{bi}$, they preserve $\mathfrak{sp}(2)\cap\mathfrak{h}$ as well. Thus $\mathfrak{k}_s=\mathrm{Lie}(K_s)$ is contained in the normalizer $$\mathfrak{k}=\mathbb{R}e_1+\mathbb{R}e_2+\mathfrak{g}_{\pm(e_1+e_2)}+
\mathfrak{g}_{\pm(e_1-e_2)}$$ of $\mathfrak{sp}(2)\cap\mathfrak{h}$ in $\mathfrak{sp}(2)$. Therefore $\dim K_s\leq 6$, and the equality holds if and only if
$\mathfrak{k}_s$ coincides with $\mathfrak{k}$.

It is obvious that $\mathfrak{c}=\mathrm{Lie}(C)=\mathbb{R}e_1+\mathbb{R}e_2+
\mathfrak{g}_{\pm e_2}$, i.e., $\dim C=4$. If $\dim K_s <6$, then $\dim C+\dim K_s<10=\dim \mathrm{Sp}(2)$. So in this case, the $C\times K_s$-action on $\mathrm{Sp}(2)$ is not transitive. If $\dim K_s=6$, i.e., $\mathfrak{k}_s=\mathfrak{k}$, then
$\mathfrak{c}\cap\mathfrak{k}_s=\mathbb{R}e_1
+\mathbb{R}e_2$ has a positive dimension, i.e., $\dim (C\cap K_s)>0$. The $C\times K_s$-orbit of the identity, $(C\times K_s)\cdot e=CK_s$, is equivariantly diffeomorphic to $C\times K_s/\Delta(C\cap K_s)$,  where $ \Delta(C\cap K_s)=\{(g,g)\ |\ \forall g\in C\cap K_s\}$. Then we have
$$\dim CK_s=\dim C+\dim K_s-\dim C\cap K_s=8<10.$$
So in this case, the $C\times K_s$-action on $\mathrm{Sp}(2)$ is not transitive either.
This completes the  proof of the lemma for $s>1$. For $0<s<1$, the proof is similar.
\end{proof}

For the convenience in later discussion, we will use quaternionic matrices to denote elements in $\mathrm{Sp}(2)$ and $\mathfrak{sp}(2)$, i.e.,
$$\mathrm{Sp}(2)=\{g\in\mathbb{H}^{2\times2}\ |\ \overline{g}^t g=e=\mathrm{diag}(1,1)\}\quad\mbox{and}\quad
\mathfrak{sp}(2)=\{u\in\mathbb{H}^{2\times2}\ |\ \overline{u}^t+ u=0\}, $$
in which
$\mathbb{H}=\mathbb{R}+\mathbb{R}\mathbf{i}+\mathbb{R}\mathbf{j}+\mathbb{R}\mathbf{k}$
is the ring of quarternions.
We choose the Cartan subalgebra $\mathfrak{t}$ of $\mathfrak{g}=B_3$ such that $\mathfrak{t}\cap\mathfrak{g}'=\{\mathrm{diag}(a\mathbf{i},\mathbb{R}\mathbf{i})\ |\
a,b\in\mathbb{R}\}$ of $\mathfrak{g}'$.
Up to a suitbale scalar, we may present $v=e_1$ as
$v=\mathrm{diag}(\mathbf{i},\mathbf{i})\in\mathfrak{sp}(2)$, and assume that the bi-invariant inner product
$\langle\cdot,\cdot\rangle_1$ on $\mathfrak{sp}(2)$ satisfies
$$\langle u,u\rangle_1=\tfrac12\mathrm{tr}(\overline{u}^t u),\quad\forall u\in\mathfrak{sp}(2).$$
We have observed in Example \ref{example-1} that,
for each $s>0$, $v$ induces a right invariant CK-vector field on
$(\mathrm{Sp}(2),\langle\cdot,\cdot\rangle_s)$.  In particular, if $s=1$, then the length of that CK-vector field is constantly $1$.

The following obvious observation will be  useful.
\bigskip

\noindent{\bf Observation 1}:
For any matrix $u$ in $\mathrm{Sp}(2)$ or $\mathfrak{sp}(2)$, the following statements are equivalent:
\begin{enumerate}
\item $u$ commutes with $v$;
\item $u$ commutes with $\exp(tv)$ for some $t\in(-\pi,0)\cup(0,\pi)$;
\item $u$ commutes with $\exp(tv)$ for all $t\in\mathbb{R}$.
\end{enumerate}
\bigskip

\begin{lemma}\label{lemma-9}
Given  $t_0\in(0,\pi)$ ($t_0\in(-\pi,0)$) and $g\in \mathrm{Sp}(2)$,
 with respect to the bi-invariant metric $\langle\cdot,\cdot\rangle_1$, the curve $c(t)=\exp (t\mathrm{Ad}(g)v)$,  $0\leq t\leq t_0$ (resp. ($c(t)=\exp (-t\mathrm{Ad}(g)v)$),  $0\leq t\leq -t_0$) is the
unique shortest unit speed geodesic from $e$ to $\exp (t_0\mathrm{Ad}(g)v)$.
\end{lemma}

\begin{proof}
 Since the metric is bi-invariant, it is enough to  prove the lemma for  $g=e$. We will only prove the assertion for $t_0\in(0,\pi)$, because the proof for the other case is similar.

First, by the bi-invariance of the metric, and the fact that
$\langle v,v\rangle_1=1$, the curve
$c(t)=\exp (tv)$ is a unit speed geodesic.

Next, we prove that $c(t)$, $0\leq t\leq t_0$ is the shortest geodesic from $e$ to $\exp(t_0v)$,  hence it  realizes the distance $t_0$
between $e$ and  $\exp(t_0v)$.
Suppose $c_1(t)$,  $0\leq t\leq t_1$ is a
unit speed geodesic from $e$ to $\exp t_0 v$, with length  $t_1$. Then by the  conjugation theorem we can write
$$c_1(t)=\exp(t \mathrm{Ad}(g_1)\mathrm{diag}(a\mathbf{i},b\mathbf{i}))=
g_1(\exp(t\mathrm{diag}(a\mathbf{i},b\mathbf{i}))g_1^{-1},$$
where $a\geq b\geq0$,  $a^2+b^2=2$, and $g_1$ is an element in $\mathrm{Sp}(2)$.
Thus  we have
\begin{equation}\label{003}
\mathrm{diag}(\exp(t_0\mathbf{i}),\exp(t_0\mathbf{i}))=c(t_0)=c_1(t_1)=
g_1\mathrm{diag}(\exp(at_1\mathbf{i}),\exp(bt_1\mathbf{i}))g_1^{-1}.
\end{equation}
Through the canonical imbedding $\mathbb{H}^{2\times 2}\hookrightarrow\mathbb{C}^{4\times 4}$,
the eigenvalues of the right side of (\ref{003}), i.e., $\exp(\pm at_1\sqrt{-1})$ and $\exp(\pm bt_1\sqrt{-1})$, must coincide with those
of the left side, i.e., $\exp(\pm t_0\sqrt{-1})$. Therefore
we  have
$at_1\equiv \pm t_0$ (mod $2\pi$),
and $bt_1\equiv \pm t_0$ (mod $2\pi$). The condition $t_0\in(0,\pi)$ implies that
$at_1\geq t_0$ and $bt_1\geq t_0$, hence   $t_1=\sqrt{\tfrac{a^2+b^2}{2}t_1^2}\geq\sqrt{\tfrac{2t_0^2}{2}}=t_0$.
Thus  $c(t)$, $0\leq t\leq t_0$ is the shortest geodesic from $e$ to $\exp(t_0v)$.

Finally, we prove the uniqueness. Suppose $c_1(t)$, $0\leq t\leq t_1$ is an unit speed shortest geodesic from $e$ to $\exp(t_0v)$. Then $t_1=t_0$, and above argument implies  $a=b=1$.
 Moreover, $g_1$ commutes with
$\mathrm{diag}(\exp(t_0\mathbf{i}),\exp(t_0\mathbf{i}))$. By Observation 1, $g_1$ commutes with $\mathrm{diag}(\mathbf{i},\mathbf{i}))$. Thus  $c_1(t)\equiv\exp(tv)\equiv c_0(t)$.
This completes the proof of the lemma.
\end{proof}

We fix a positive integer $n$, and denote by $\Gamma$ the subgroup $\mathbb{Z}_{2n+1}$ of the one-parameter subgroup $\exp(\mathbb{R}v)\subset \mathrm{Sp}(2)$.  It is easily seen that the elements of
$\Gamma$ can be expressed as $$\exp(\tfrac{k\pi}{2n+1}v)=\mathrm{diag}(\exp(\tfrac{k\pi}{2n+1}\mathbf{i}),
 \exp(\tfrac{k\pi}{2n+1}\mathbf{i})),\quad  k=-n,-n+1,\cdots,n-1,n.
$$
Then we have the Riemannian quotients $(\Gamma\backslash \mathrm{Sp}(2),\langle\cdot,\cdot\rangle_s)$, i.e., for each $s>0$, the quotient map from $(\mathrm{Sp}(2),\langle\cdot,\cdot\rangle_s)$ to
$(\Gamma\backslash\mathrm{Sp}(2),\langle\cdot,\cdot\rangle_s)$ is locally isometric.

\begin{lemma}\label{lemma-11}
If $s>0$ and $s\neq1$, then the Riemannian quotient $(\Gamma\backslash G',\langle\cdot,\cdot\rangle_s)$ is not  homogeneous.
\end{lemma}
\begin{proof}
It is obvious that connected isometry group
$I_0(\Gamma\backslash \mathrm{Sp}(2),\langle\cdot,\cdot\rangle_s)$
is locally isomorphic to the centralizer of $L(\Gamma)$ in $I_0(\mathrm{Sp}(2),\langle\cdot,\cdot\rangle_s)$. We may present $I_0(\mathrm{Sp}(2),\langle\cdot,\cdot\rangle_s)$ as $L(\mathrm{Sp}(2))R(K_s)$, in which
$R(K_s)$ always commutes with $L(\Gamma)$, and by Observation 1,
the centralizer of $\Gamma=L(\Gamma)$ in $\mathrm{Sp}(2)=L(\mathrm{Sp}(2))$ coincides with the
centralizer $C$ of $v$. Thus the Lie transformation group $I_0(\mathrm{Sp}(2),\langle\cdot,\cdot\rangle_s)$
is locally isomorphic to $C\times K_s$, such that $C$ and $K_s$ acts on $\mathrm{Sp}(2)$ by left and right translations respectively. Now Lemma \ref{lemma-7} indicates that this action is not transitive when $s\neq1$, i.e., $(\Gamma\backslash \mathrm{Sp}(2),\langle\cdot,\cdot\rangle_s)$ is not homogeneous in this situation.
\end{proof}

\begin{lemma}\label{lemma-12}
Let $\rho$ be the left translation $L(\exp (t_0v))$ on $\mathrm{Sp}(2)$
with $t_0\in(-\pi,\pi)$. Then $\rho$ is a CW-translation on $(\mathrm{Sp}(2),\langle\cdot,\cdot\rangle_s)$ when $s$ is sufficiently close to 1.
\end{lemma}

\begin{proof}We will only prove the case  $t_0\in(0,\pi)$, since the case $t_0\in(-\pi, 0)$ is similar, and the case $t_0=0$ is trivial. Denote by $\mathrm{Exp}_s:\mathfrak{sp}(2)=T_e\mathrm{Sp}(2)\rightarrow \mathrm{Sp}(2)$ the exponential map at $e\in \mathrm{Sp}(2)$, with respect to the metric $\langle\cdot,\cdot\rangle_s$. Then by the properties of the exponential map in Riemannian geometry, and  Lemma \ref{lemma-9}, there exists a positive smooth function $r(w)$ on the unit sphere $\mathcal{S}=\{w|
\langle w,w\rangle_1=1\}\subset\mathfrak{sp}(2)$ satisfying the following two requirements:
\begin{enumerate}
\item for any $w\in\mathcal{S}$, the curve $t\in[0,r(w)]\mapsto \exp(tw)$ is the unique unit speed shortest geodesic from $e$ to $\exp(r(w)w)$ on $(\mathrm{Sp}(2),\langle\cdot,\cdot\rangle_1)$;
\item for any $u\in\mathrm{Ad}(\mathrm{Sp}(2))v\subset\mathcal{S}$, we have $r(u)>t_0$.
\end{enumerate}

Obviously, $\mathrm{Exp}_1$ is a diffeomorphism from
the open disk $\mathcal{D}=\bigcup_{u\in\mathcal{S}_1}\{tu|\forall t\in[0,r(u))\}\subset\mathfrak{g}'$ to its image. This observation is still valid with the metric slightly perturbed. To be precise, we have the following claim, which is easy to prove.
\bigskip

\noindent{\bf Claim A}. If the parameter $s$ in $\langle\cdot,\cdot\rangle_s$
is sufficiently close to $1$, then
$\mathrm{Exp}_s$ is a diffeomorphism from $\mathcal{D}$ to its image $\mathrm{Exp}_s(\mathcal{D})\subset \mathrm{Sp}(2)$.
\bigskip


Next, we prove

\noindent{\bf Claim B}. If $s$ is sufficiently close to $1$, then for any $u\in\mathrm{Ad}(\mathrm{Sp}(2))v$, the curve
$c(t)=\exp tu$ with $t\in[0,t_0]$ is a shortest geodesic from $e$ to $\exp(t_0u)$ on $(\mathrm{Sp}(2),\langle\cdot,\cdot\rangle_s)$, and its   length is
$t_0\langle v,v\rangle_s{}^{1/2}$.

We have seen in Example \ref{example-1} that,
for the left invariant metric $\langle\cdot,\cdot\rangle_s$, the vector $v=\mathrm{diag}(\mathbf{i},\mathbf{i})\in\mathfrak{sp}(2)$ generates a right invariant CK-vector field on $\mathrm{Sp}(2)$.
Then by Lemma \ref{lemma-6}, the vector $u\in\mathrm{Ad}(\mathrm{Sp}(2))v$ also generates a right invariant CK-vector field. So $c(t)=\exp tu$,  $0\leq t\leq t_0$ is a geodesic of length $t_0\langle v,v\rangle_s^{1/2}$ on $(\mathrm{Sp}(2),\langle\cdot,\cdot\rangle_s)$.

Now we prove that the curve $c(t)$,  $0\leq t\leq t_0$ is shortest when $s$ is sufficiently close to $1$.
Assume conversely that this is not true. Then there exists a sequence of positive numbers $s_n$ approaching $1$, a sequence $u_n$ with $(t_0\langle v,v\rangle_{s_n}^{1/2})^{-1}{u_n}\in\mathrm{Ad}(\mathrm{Sp}(2))v$,
and a sequence $w_n\in\mathfrak{sp}(2)$ with $\langle w_n,w_n\rangle_{s_n}
< \langle u_n,u_n\rangle_{s_n}=t_0^2\langle v,v\rangle_{s_n}$,
such that $\mathrm{Exp}_s (w_n)=\mathrm{Exp}_s(u_n)=\exp u_n$. Obviously, for any $n$, $u_n$ and $w_n$ are different.

Since $\mathop{\lim}\limits_{n\rightarrow\infty}\langle v,v\rangle_{s_n}=1$, the sequences $\{u_n\}$ and $\{w_n\}$ are both bounded. Replacing these sequences with suitable subsequences if necessary,  we can assume that $u_n$ and $w_n$ converge to $u_0$ and $w_0$, respectively. By the continuity, we have $t_0^{-1}{u_0}\in\mathrm{Ad}(\mathrm{Sp}(2))v$,
$\langle w_0,w_0\rangle_1\leq t_0^2$ and $\exp( u_0)=\mathrm{Exp}_1(u_0)=\mathrm{Exp}_1(w_0)=\exp(w_0)$. Then Lemma \ref{lemma-9} provides $u_0=w_0$. Since $t_0<r(t_0^{-1}u_0)$,
the  continuity implies that $\langle u_n,u_n\rangle_1^{1/2}<r(u_n)$ and $\langle w_n,w_n\rangle_1^{1/2}<r(w_n)$ for sufficiently large $n$. So when $n$ is sufficiently large, $s_n$ is sufficiently close to 1, and the pair $\{u_n,w_n\}$ are
two different points in $\mathcal{D}$ which are mapped by $\mathrm{Exp}_{s_n}$ to the same point. This contradicts Claim A.
The proof of Claim B is finished.

To summarize, Claim B implies that, when $s$ is sufficiently close to 1, the distance from $e$ to $\exp(t_0\mathrm{Ad}(g^{-1})v)=g^{-1}\exp(t_0v)g$ is a constant, for all  $g\in(\mathrm{Sp}(2),\langle\cdot,\cdot\rangle_s)$. Then by the left invariance of the metric, the distance from $g$ to $\exp(t_0v)g$ is also a constant function. So
$\rho$ is a CW-translation on $(\mathrm{Sp}(2),\langle\cdot,\cdot\rangle_s)$, when $s$ is sufficiently close to $1$.
\end{proof}

Applying Lemma \ref{lemma-12} repeatedly, we can make the left translation of each element in $\Gamma$ a CW-translation on $(\mathrm{Sp}(2),\langle\cdot,\cdot\rangle_s)$
when $s\neq1$ is sufficiently close to 1. Then using Lemma \ref{lemma-11}, we get the first counter example to the Riemannian homogeneity conjecture, which can be summarized as follows.

\begin{theorem}\label{thm-2}
Let $n$ be a positive integer, and $\Gamma$ the subgroup $\mathbb{Z}_{2n+1}$ of $\mathrm{Sp}(2)$ generated
by $\mathrm{diag}(\exp(\tfrac{\pi}{2n+1}\mathbf{i}),\exp(\tfrac{\pi}{2n+1})\mathbf{i})$.
Then there exist
a smooth family of left invariant Riemannian metrics $\langle\cdot,\cdot\rangle_s$, $s>0$ on $\mathrm{Sp}(2)$ such that
\begin{enumerate}
\item the metric $\langle\cdot,\cdot\rangle_1$ is bi-invariant;
\item when $s$ is sufficiently close to $1$, the left translation  defined by any  element in $\Gamma$ is a CW-translation of  $(\mathrm{Sp}(2),\langle\cdot,\cdot\rangle_s)$;
\item when $s\neq 1$, the Riemannian quotient $(\Gamma\backslash \mathrm{Sp}(2),\langle\cdot,\cdot\rangle_s)$ is not homogeneous.
\end{enumerate}
\end{theorem}

\bigskip


\end{document}